\documentclass[a4paper,12pt]{article}
\usepackage{hyperref}
\hypersetup{
    colorlinks=false,
    linktoc=all,
    linkcolor=red,     %choose some color if you want links to stand out
    citecolor=blue,     %choose some color if you want links to stand out
}

\usepackage{amsmath,amssymb,amsfonts,amsthm,graphics,enumitem}
\usepackage{latexsym, bm}
\usepackage{multicol}
\usepackage{indentfirst}

\usepackage{cite}

\usepackage{graphicx}
\usepackage{caption}
\usepackage{textcomp}
\usepackage{tikz}
\usetikzlibrary{calc,arrows,positioning,shapes}
\usepackage{bigdelim}
\usepackage{nicematrix}
\usepackage{mathtools}
\usepackage{comment}
\usepackage{makecell}
\usetikzlibrary{decorations.pathreplacing}
\usetikzlibrary{intersections}
\usepackage{enumerate}
\usepackage{graphicx,booktabs,multirow}
\usepackage{appendix}

\newcommand{\keywords}[1]{%
  \par\medskip
  \noindent{\small\textbf{Keywords:} #1\par}
  \medskip
}

\newcommand{\msc}[1]{%
  \par\smallskip
  \noindent{\small\textbf{MSC:} #1\par}
  \medskip
}

\newcommand{\Ex}{\operatorname{Ex}}
\newcommand{\Exsp}{\operatorname{Ex}_{\mathrm{sp}}}
\newcommand{\Exssp}{\operatorname{Ex}_{\mathrm{ssp}}}

\newcommand{\ex}{\operatorname{ex}}
\newcommand{\exsp}{\operatorname{ex}_{\mathrm{sp}}}
\newcommand{\exssp}{\operatorname{ex}_{\mathrm{ssp}}}

\usepackage{setspace}
\newtheorem{theorem}{Theorem}[section]

\newtheorem{Observation}[theorem]{\rm\bfseries Observation}
\newtheorem{lemma}[theorem]{Lemma}

\newtheorem{Proof of Theorem 1.7.}[theorem]{Proof of Theorem 1.7.}

\newtheorem{definition}[theorem]{Definition}
\newcommand{\ignore}[1]{}
\begin{document}
\noindent
\begin{spacing}{1}

\title {The signless Laplacian spectral radius of graphs without disjoint cliques}
\date{}

\author{
Xinghui Zhao\footnote{	School of Mathematical Sciences, South China Normal University,  Guangzhou, 510631, P. R. China. E-mail: {\tt xhzhao@m.scnu.edu.cn}. } \;\;   \;\; Lihua You\footnote{Corresponding author. School of Mathematical Sciences, South China Normal University, Guangzhou, 510631, P. R. China. 
E-mail: {\tt ylhua@scnu.edu.cn}.} \;\;   \;\; Jing Zeng\footnote{College of Cryptology and Cyber Science, Nankai University, Tianjin, 300350, P. R. China. E-mail: \texttt{jingzeng@mail.nankai.edu.cn}.}
}
\maketitle
\begin{abstract}
A graph $G$ is $(t+1)K_{r+1}$-free if it contains no $t+1$ pairwise vertex-disjoint copies of $K_{r+1}$. 
Moon [Canad. J. Math. 20 (1968) 95-102] and Simonovits [Theory of Graphs (Proc. Colloq., Tihany, 1966)] independently determined that, for sufficiently large $n$, $K_{t}\vee T_{r}(n-t)$ is the unique $n$-vertex $(t+1)K_{r+1}$-free graph with the maximum number of edges.
In 2023, Ni, Wang and Kang [Electron. J. Combin. 30 (2023) \#P1.20] showed that the graph $K_{t}\vee T_{r}(n-t)$ is also the unique adjacency spectral extremal graph over all $n$-vertex $(t+1)K_{r+1}$-free graphs for sufficiently large $n$.
In this paper, for $r\geq 3$ and $t\geq 0$, we prove that $K_{t}\vee T_r(n-t)$ is the unique graph attaining the maximum signless Laplacian spectral radius among all $(t+1)K_{r+1}$-free graphs of sufficiently large order $n$.
\end{abstract}

\msc{05C35, 05C50}

\section{Introduction}
As one of the central areas of graph theory, extremal graph theory has attracted considerable attention and has been extensively studied from various perspectives. A cornerstone of extremal graph theory is Tur\'an’s theorem \cite{Turan1941}, which states that the maximum number of edges in an $n$-vertex $K_{r+1}$-free graph is attained uniquely by the balanced complete $r$-partite graph. This graph is known as the Tur\'an graph and is denoted by $T_r(n)$. A classical Tur\'an-type problem asks for the maximum number of edges in an $n$-vertex graph that does not contain a given graph $F$ as a subgraph. In 2010, Nikiforov \cite{Nikispecturan} proposed the spectral version of this problem, also known as the spectral Tur\'an problem, which asks for the maximum spectral radius of an $F$-free graph of order $n$. A well-known result in this direction is Nikiforov's spectral Tur\'an theorem \cite{Nikituranthm}, which states that among all $K_{r+1}$-free graphs of order $n$, the Tur\'an graph $T_r(n)$ uniquely maximizes the spectral radius of its adjacency matrix. The corresponding result for the signless Laplacian spectral radius was later obtained by He, Jin and Zhang \cite{He}, who proved that, for $r\geq 3$, the Tur\'an graph $T_r(n)$ also uniquely maximizes the signless Laplacian spectral radius over all $K_{r+1}$-free graphs of order $n$. For various forbidden subgraphs, many advances have been made on the corresponding spectral extremal problems for both the adjacency matrix and the signless Laplacian matrix; see, e.g., \cite{ChenLiuZhangforest,SCioab,CioabDMevenC,friendship,XYuan1,XYuan2,XYuan3,WangZhaifan,ZhaiLinC6,ZhaiLinbook,ZhaiWangC4,ZhaoHGtri}. In this paper, we investigate the signless Laplacian spectral Tur\'an problem for $(t+1)K_{r+1}$-free graphs.

Throughout this paper, we consider only finite, undirected and simple graphs. The signless Laplacian matrix of a graph $G$, denoted by $Q(G)$, is defined as $Q(G)=A(G)+D(G)$, where $A(G)$ and $D(G)$ denote the adjacency matrix and the diagonal degree matrix of $G$, respectively. For a graph $F$, let $\Ex(n,F)$, $\Exsp(n,F)$ and $\Exssp(n,F)$ be the sets of $n$-vertex $F$-free graphs maximizing the number of edges, adjacency spectral radius and signless Laplacian spectral radius, respectively. We denote by $\ex(n,F)$, $\exsp(n,F)$ and $\exssp(n,F)$ the corresponding extremal values. Let $F_{k,r+1}$ be the graph consisting of $k$ copies of $K_{r+1}$ sharing exactly one common vertex. It is known that, when $n$ is sufficiently large, $\Exsp(n,F)\subseteq \Ex(n,F)$ holds for several graphs $F$, such as $K_{r+1}$ \cite{Nikituranthm}, $F_{k,r+1}$ \cite{DNDesai}, $F_k$ \cite{friendship} and $H_{s,k}$ \cite{LiPeng}, where $F_k$ is the friendship graph and $H_{s,k}$ is the graph consisting of $s$ triangles and $k$ odd cycles of length at least five sharing exactly one common vertex. Cioab\u{a}, Desai and Tait \cite{SCioab} conjectured that if the graphs in $\Ex(n,F)$ are Tur\'an graphs plus $O(1)$ edges, then $\Exsp(n,F)\subseteq \Ex(n,F)$ when $n$ is sufficiently large. Wang, Kang and Xue \cite{WKX23} confirmed this conjecture and proved the stronger result that the conclusion remains true provided that $\ex(n,F)=e(T_{r}(n))+O(1)$. Further developments of this result were obtained by Fang, Tait and Zhai \cite{FangTaitZhaiDM} and Byrne \cite{Byrne}. Very recently, Chen, Jin, Zhang and Zhang \cite{Cen1} proved the corresponding result for the signless Laplacian spectral radius. They showed that if chromatic number $\chi(F)=r+1\geq4$ and $\ex(n,F)=e(T_{r}(n))+O(1)$, then $\Exssp(n,F)\subseteq \Ex(n,F)$ when $n$ is sufficiently large.

For disjoint cliques, both $\Ex(n,(t+1)K_{r+1})$ and $\Exsp(n,(t+1)K_{r+1})$ have been determined for all sufficiently large $n$. Simonovits \cite{Simonovits} and Moon \cite{Moon} showed that $K_{t}\vee T_r(n-t)$ is the unique extremal graph in $\Ex(n,(t+1)K_{r+1})$, where $n$ is sufficiently large.
Note that $e\bigl(K_{t}\vee T_r(n-t)\bigr)=e(T_{r}(n))+\frac{t}{r}n+O(1)$. Hence one cannot obtain $\Exsp(n,(t+1)K_{r+1})\subseteq \Ex(n,(t+1)K_{r+1})$
directly from the result of Wang, Kang and Xue. However, in 2023, Ni, Wang and Kang \cite{SPEXtKr} proved that $K_{t}\vee T_r(n-t)$ is also the unique extremal graph in $\Exsp(n,(t+1)K_{r+1})$ for sufficiently large $n$.

Motivated by the results of Simonovits \cite{Simonovits}, Moon \cite{Moon}, and Ni, Wang and Kang \cite{SPEXtKr}, a natural problem is to determine $\Exssp(n,(t+1)K_{r+1})$. To date, some results are known for certain pairs $(t+1,r+1)$. Zhang and Wang solved the cases $(t+1,r+1)=(2,3)$ for $n\geq 44$ \cite{ZhangWang2K3} and $(t+1,r+1)=(3,3)$ for $n\geq 453$ \cite{ZhangWang3K3}. Subsequently, Zhang and Lei \cite{Loworder2K3} settled the remaining range $6\leq n\leq 43$ for the case $(t+1,r+1)=(2,3)$. Moreover, Zheng, Li and Fan \cite{ZhengLiFan26} showed that if $t\geq 2$, $\varepsilon>0$ and $G$ is an $n$-vertex graph with $q(G)\geq \left(1-\frac{1}{t}+\varepsilon\right)2n$, then $G$ contains $\Omega_{t,\varepsilon}(n^{t+1})$ copies of $K_{t+1}$. Recently, Zeng \cite{Zeng} proved that $K_{t}\vee K_{\lfloor (n-t)/2\rfloor,\lceil (n-t)/2\rceil}$ is the unique graph in $\Exssp(n,(t+1)K_3)$ for all $t\geq 1$ and $n\geq 28t+11$. However, as noted in \cite{Loworder2K3}, the extremal graph with the maximum signless Laplacian spectral radius is still unknown for general $(t+1,r+1)$.

In this paper, we determine the unique graph with maximum signless Laplacian spectral radius among all $n$-vertex $(t+1)K_{r+1}$-free graphs for all $r\geq 3$ and $t\geq 0$, whenever $n$ is sufficiently large.
\begin{theorem}\label{t}
    Let $r\geq 3$, $t\geq0$ be fixed integers, and let $n$ be sufficiently large. Let $G\in Ex_{ssp}(n,(t+1)K_{r+1})$. Then  $G\cong K_t\vee T_r(n-t)$.
\end{theorem}

Combining our result with the result of Zeng \cite{Zeng} on $(t+1)K_3$ and the result of He, Jin and Zhang \cite{He} on $K_{r+1}$, it follows that $\Exssp(n,(t+1)K_{r+1})=\Ex(n,(t+1)K_{r+1})$ for all $r\geq 2$ and $t\geq 0$, except for $(r,t)=(2,0)$, whenever $n$ is sufficiently large.

\subsection{Notation}\label{subsec-nota}

 Throughout this paper, we use $q(G)$ to denote the maximum eigenvalue of $Q(G)$. Let $V(G)$ and $E(G)$ be the vertex set and the edge set of a graph $G$, respectively, and write $e(G)=|E(G)|$. For a graph $G$, we use $\overline{G}$ to denote its complement. 
For a set $E\subseteq E(\overline{G})$, let $G+E$ be the graph obtained by adding all edges in $E$, and for a set $V\subseteq V(G)$, let $G-V$ be the graph obtained by deleting the vertices in $V$ (and their incident edges). When $E=\{e\}$, we simply write $G+e$. For convenience, we write $\{1,\ldots,n\}$ as $[n]$.
 
\subsection{Organization}\label{subsec-org}

In Section \ref{sec-pre}, we present some necessary preliminary results. In Section \ref{sec-pf}, we first establish several lemmas that are crucial to the proof, and then prove Theorem \ref{t}.

	\section{Preliminaries}\label{sec-pre}

    \begin{lemma}{\rm(\!\!\cite{Cve})}\label{ll1}
        Let $G$ be a connected graph. If $H$ is a proper subgraph of $G$, then $q(H)<q(G)$.
    \end{lemma}

    \begin{lemma}{\rm(\!\!\cite{Cai1,Cen1})}\label{l1}
       Let $n, b,$ and $r$ be non-negative integers satisfying $n = br + s$, where $b\geq1$ and  $0 \leq s < r$. Then $q(T_r(n))=n-2b+\frac{n-2+\sqrt{(n-2)^2+8b(r-s)}}{2}$.
    \end{lemma}

     \begin{lemma}{\rm(\!\!\cite{He})}\label{ll2}
       Let $r\ge3$, and let  $G\in Ex_{ssp}(n,K_{r+1})$. Then  $G\cong  T_r(n)$.
    \end{lemma}

    \begin{lemma}{\rm(\!\!\cite{He})}\label{ll6}
         Let $a\ge1$ be an integer, and let   $R$ be a  graph  with $|V(R)|\geq1$. Let $H$ be a  $K_{a+1}$-free graph. Then there exists a graph $H_1$ of order $m$ such that $H_1=bK_1\vee H_2$ and $q(R\vee H)\leq q(R\vee H_1)$, where $H_2$ is a $K_a$-free graph and $b\geq1$. Moreover, $q(R\vee H)=q(R\vee H_1)$ if and only if $H\cong H_1$.
        
    \end{lemma}

\begin{lemma}{\rm(\!\!\cite{Cen1})}\label{l2}
    Let $r\geq3$ be a fixed integer and $F$ be any graph such that chromatic number $\chi(F)=r+1$ and $ex(n, F)= e(T_r(n))+O(1)$. Then $Ex_{ssp}(n,F)\subseteq Ex(n,F)$ for sufficiently large $n$.
\end{lemma}

Recall that $F_{k,r+1}$ is the graph consisting of $k$ copies of $K_{r+1}$ sharing exactly one common vertex.

\begin{lemma}{\rm(\!\!\cite{Cen2,Hou})}\label{l3}
   Let $r\ge3$ be a fixed integer and $n$ be sufficiently large. Suppose $G\in Ex(n,F_{k,r+1})$. Then there exists a subset  $E_0\subseteq E(\overline{T_{r}(n)})$ with $|E_0|=O(1)$ such that  $G \cong T_{r}(n) + E_0$.
\end{lemma}

By $\chi(F_{k,r+1})=r+1$, Lemma \ref{l2} and Lemma \ref{l3}, we immediately have the following result.

\begin{theorem}\label{l4}
    Let $r\ge3,k\ge1$ be fixed integers and $n$ be sufficiently large. Suppose $G\in Ex_{ssp}(n,F_{k,r+1})$. Then there exists a subset  $E_0\subseteq E(\overline{T_{r}(n)})$ with $|E_0|=O(1)$ such that  $G \cong T_{r}(n) + E_0$.
\end{theorem}
  \section{Proof of Theorem \ref{t}} \label{sec-pf}
Before proving Theorem  \ref{t}, we present some important lemmas.
In the rest of this paper, we set $\alpha = \frac{2(r-1)}{r}$ and $\beta = \frac{4(r-1)}{r(r-2)}$ throughout.
    \begin{lemma}\label{l5}
        Let $r\geq 3$ be a fixed integer, and let $n$ be sufficiently large. Then $q(T_r(n))=\alpha n+O(\frac{1}{n})$.
    \end{lemma}
	\begin{proof}
	    Let $0 \leq s < r$ and $n \equiv  s\ (\mathrm{mod}\ r)$. Then  $q(T_r(n))=n-\frac{2(n-s)}{r}+\frac{n-2+\sqrt{(n-2)^2+\frac{8(n-s)(r-s)}{r}}}{2}$ by Lemma \ref{l1}. By direct calculation, we have $(n-2)^2+\frac{8(n-s)(r-s)}{r}=n^2+(4-\frac{8s}{r})n+4-\frac{8s(r-s)}{r}$. Let $A=4-\frac{8s}{r}$ and $B=4-\frac{8s(r-s)}{r}$. Then $\sqrt{(n-2)^2+\frac{8(n-s)(r-s)}{r}}=n\sqrt{1+\frac{A}{n}+\frac{B}{n^2}}$ and $\frac{A}{n}+\frac{B}{n^2}=O(\frac{1}{n})$. Let $f(x)=\sqrt{1+x}$. Then the Taylor expansion of $f(x)$ at $x = 0$ is $f(x) = 1 + \frac{x}{2}+O(x^2)$. Thus we have \begin{align*}
	        &\sqrt{(n-2)^2+\frac{8(n-s)(r-s)}{r}}\\=&n(1+\frac{1}{2}(\frac{A}{n}+\frac{B}{n^2})+O(\frac{1}{n^2}))\\=&n+2-\frac{4s}{r}+O(\frac{1}{n}).
	    \end{align*}

        Therefore, \begin{align*}
            q(T_r(n))&=n-\frac{2(n-s)}{r}+\frac{(n-2)+(n+2-\frac{4s}{r}+O(\frac{1}{n}))}{2}\\&=\frac{2(r-1)}{r}n+O(\frac{1}{n})=\alpha n+O(\frac{1}{n})
        \end{align*}
        as desired.
	\end{proof}

	For a connected graph $G$, we know that there exists a positive eigenvector $\textbf{x}=(x_v)_{v\in V(G)}$  corresponding to $q(G)$ by the Perron-Frobenius Theorem. We refer to such a vector $\textbf{x}=(x_v)_{v\in V(G)}$ as a Perron vector of $Q(G)$.

\begin{lemma}\label{l6}
    Let $r\geq 3, t\geq0$ be fixed integers, and let $n$ be sufficiently large. Then $q(K_t\vee T_r(n-t))=q(T_r(n))+t\beta+O(\frac{1}{n})$.
\end{lemma}
\begin{proof}
    Let $T_r(n-t)=K_{a_1,\ldots,a_r}$, where $\sum\limits_{i=1}^{r}a_i=n-t$ and $|a_i-a_j|\leq 1$ for any $1\leq i,j\leq r$.

    If $t=0$, then the result is as follows by Lemma \ref{l5}. 
    
    Now we assume that $t\geq1$. Let $\textbf{x}=(x_v)_{v\in V(K_t\vee T_r(n-t))}$ be the unit Perron vector of $Q(K_t\vee T_r(n-t))$. Then we may assume that the entries of the Perron vector corresponding to the vertices  in $K_t$ are $x_0$, and the entries of the Perron vector corresponding to the vertices in the $i$-th part of $K_{a_1,\ldots,a_r}$ are $x_i$, where $i\in [r]$. Let $q\equiv q(K_t\vee T_r(n-t))$ and $X=tx_0+\sum\limits_{i=1}^{r}a_ix_i$.  Since $q\textbf{x}=Q(K_t\vee T_r(n-t))\textbf{x}$, we have 

    \begin{align}
        qx_0=(n-1)x_0+(t-1)x_0+\sum\limits_{i=1}^{r}a_ix_i\label{s1}
    \end{align}
and \begin{align}
     qx_i=(n-a_i)x_i+tx_0+\sum\limits_{j\in [r]\setminus \{i\}}a_jx_j\label{s2}
\end{align}
for each $i\in [r]$. 
From  \eqref{s1} and \eqref{s2}, we have $(q-n+2)x_0=X$ and $(q-n+2a_i)x_i=X$ for  each $i\in [r]$, respectively. Then 
    $X=tx_0+\sum\limits_{i=1}^{r}a_ix_i=(\frac{t}{q-n+2}+\sum\limits_{i=1}^{r}\frac{a_i}{q-n+2a_i})X$, 
    and thus \begin{align}
        1=\frac{t}{q-n+2}+\sum\limits_{i=1}^{r}\frac{a_i}{q-n+2a_i}.\label{s7}
    \end{align}

   By the Rayleigh quotient and Lemma \ref{ll1}, we have $\frac{\mathbf{1}^T Q(K_t\vee T_r(n-t)) \mathbf{1}}{\mathbf{1}^T\mathbf{1}} \le q \le q(K_n)$, where $\mathbf{1}$ is an all-ones vector. Then $\frac{2(r-1)}{r} n + O(1) \le q \le 2(n-1)$. Let $b_n=\frac{q}{n}$. Then $\frac{2(r-1)}{r}\le b_n\le2$. Moreover,  for any $i\in [r]$, there exists a $\delta_i=O(1)$ such that $a_i=\frac{n}{r}+\delta_i$ and $\sum\limits_{i=1}^{r}\delta_i=-t$. Then for any $i\in [r]$, we have\begin{align}
       \frac{a_i}{q-n+2a_i}=\frac{\frac{n}{r}+\delta_i}{(b_n-1+\frac{2}{r})n+2\delta_i}.\label{s8}
   \end{align}
From \eqref{s7}, \eqref{s8} and $q=b_nn$, it follows that \begin{align*}
    1&=\frac{t}{(b_n-1)n+2}+\frac{1}{r}\sum\limits_{i=1}^{r}\frac{1+\frac{r\delta_i}{n}}{b_n-1+\frac{2}{r}+\frac{2\delta_i}{n}} \nonumber\\
    &=O(\frac{1}{n})+\frac{1}{r}\sum\limits_{i=1}^{r}(\frac{1}{b_n-1+\frac{2}{r}}+O(\frac{1}{n}))\nonumber\\
    &=\frac{1}{b_n-1+\frac{2}{r}}+O(\frac{1}{n}).
\end{align*}
   Then $b_n=\frac{2(r-1)}{r}+O(\frac{1}{n})$, and thus $q=\frac{2(r-1)}{r} n+O(1)$.
   
   Let $c_n=q-\frac{2(r-1)}{r}n=O(1)$.
From $q=\frac{2(r-1)}{r}n+c_n$ and $a_i=\frac{n}{r}+\delta_i$, we have $\frac{t}{q-n+2}=\frac{t}{\frac{r-2}{r}n+c_n+2}$ and \begin{align*}
       \sum\limits_{i=1}^{r}\frac{a_i}{q-n+2a_i}&=\sum\limits_{i=1}^{r}\frac{\frac{n}{r}+\delta_i}{n+c_n+2\delta_i}\\
       &=1+\sum\limits_{i=1}^{r}\frac{\frac{r-2}{r}\delta_i-\frac{c_n}{r}}{n+c_n+2\delta_i}\\
       &=1+\sum\limits_{i=1}^{r}(\frac{\frac{r-2}{r}\delta_i-\frac{c_n}{r}}{n}+O(\frac{1}{n^2}))\\
       &=1-\frac{\frac{r-2}{r}t+c_n}{n}+O(\frac{1}{n^2}).
   \end{align*}

   Combining \eqref{s7}, it follows that \begin{align*}
       1&=\frac{t}{\frac{r-2}{r}n+c_n+2}+1-\frac{\frac{r-2}{r}t+c_n}{n}+O(\frac{1}{n^2})\\
       &=1+\frac{t}{\frac{r-2}{r}n}-\frac{\frac{r-2}{r}t+c_n}{n}+O(\frac{1}{n^2}).
   \end{align*}
   Thus $\frac{t}{\frac{r-2}{r}n}-\frac{\frac{r-2}{r}t+c_n}{n}+O(\frac{1}{n^2})=0$, which implies $c_n=\frac{4(r-1)}{r(r-2)}t+O(\frac{1}{n})$. Therefore, we have $q=\frac{2(r-1)}{r}n+\frac{4(r-1)}{r(r-2)}t+O(\frac{1}{n})=\alpha n+\beta t+O(\frac{1}{n})$.
\end{proof}

\begin{lemma}\label{l7}
    Let $r\ge3,k\ge1$ be  fixed integers and $n$ be sufficiently large. Then $ex_{ssp}(n,F_{k,r+1})=q(T_r(n))+O(\frac{1}{n})$.  
\end{lemma}
\begin{proof}
Let $G\in Ex_{ssp}(n,F_{k,r+1})$. Then  there exists a subset  $E_0\subseteq E(\overline{T_{r}(n)})$ with $|E_0|=O(1)$ such that  $G \cong T_{r}(n) + E_0$ by Theorem \ref{l4}. Thus we have $q(G)\geq q(T_{r}(n))=\alpha n+O(\frac{1}{n})$ by Lemma \ref{l5}.

 Let $\textbf{x}=(x_v)_{v\in V(G)}$ be the unit Perron vector of $Q(G)$. Then for each $v\in V(G)$, we have\begin{align*}
    (q(G)-d_G(v))x_v=\sum\limits_{u\in N_G(v)}x_u\leq\sqrt{d_G(v)\sum\limits_{u\in N_G(v)}x^2_u}\leq \sqrt{d_G(v)}.
\end{align*}

Note that $d_G(v)= (1-\frac{1}{r})n+O(1)$ for each $v\in V(G)$. Thus $$x_v\leq \frac{ \sqrt{d_G(v)}}{q(G)-d_G(v)}\leq \frac{\sqrt{(1-\frac{1}{r})n+O(1)}}{\alpha n+O(\frac{1}{n})-(1-\frac{1}{r})n-O(1)}=O(\frac{1}{\sqrt{n}})$$ for each $v\in V(G)$. By the Rayleigh quotient, we have $q(G)=\textbf{x}^TQ(G)\textbf{x}=\textbf{x}^TQ(T_r(n))\textbf{x}+\sum\limits_{uv\in E_0}(x_u+x_v)^2\leq q(T_r(n))+O(\frac{1}{n})$.

Combining the lower and upper bounds gives $ex_{ssp}(n,F_{k,r+1})=q(T_r(n))+O(\frac{1}{n})$. 
\end{proof}

\begin{lemma}\label{l8}
    Let $r\ge3,p\ge0$ be  fixed integers and $m$ be sufficiently large. Let $H$ be a graph of order $m$ that satisfies $q(H)\leq q(T_r(m))+c+O(\frac{1}{m})$, where $c$  is a constant.  Then $q(K_p\vee H)\leq q(T_r(m+p))+c+p\beta+O(\frac{1}{m})$.
\end{lemma}
\begin{proof}
    When $p=0$, the conclusion is trivial, so we assume $p\ge1$.

    Let $G=K_p\vee H$ and $\textbf{x}^T=(\textbf{x}^T_{K_p}, \textbf{x}^T_{H})=(x_v)_{v\in V(G)}$ be the unit Perron vector of $Q(G)$, where $\textbf{x}_{K_p}$ and $\textbf{x}_{H}$ correspond to $K_p$ and $H$, respectively. Moreover, let $X=\sum\limits_{v\in V(K_p)}x^2_v$ and $Y=\sum\limits_{u\in V(H)}x^2_u$. Then $X+Y=1$. Clearly, we have \begin{align*}
        q(G)&=\textbf{x}^TQ(G)\textbf{x}=\sum\limits_{uv\in E(G)}(x_u+x_v)^2\\&=\sum\limits_{uv\in E(K_p)}(x_u+x_v)^2+\sum\limits_{uv\in E(H)}(x_u+x_v)^2+\sum\limits_{u\in V(K_p),v\in V(H)}(x_u+x_v)^2\\&\leq q(K_p)X+q(H)Y+(mX+pY+2\sqrt{pmXY}\\&
        =(\sqrt{X},\sqrt{Y})\begin{pmatrix}
        m+2(p-1) & \sqrt{pm} \\
        \sqrt{pm}& q(H)+p
    \end{pmatrix}\begin{pmatrix}
        \sqrt{X}\\
        \sqrt{Y}
    \end{pmatrix}.
    \end{align*}
By Lemma \ref{l5} and  $q(H)\leq q(T_r(m))+c+O(\frac{1}{m})$, we have $q(H)\leq \alpha m+c+O(\frac{1}{m})$.  Let $A=m+2(p-1), B=\alpha m+c+p+O(\frac{1}{m})$ and $S=\begin{pmatrix}
       A & \sqrt{pm} \\
        \sqrt{pm}& B
    \end{pmatrix}$.
 Then $q(G)\leq \lambda_1(S)$. By direct calculation, we have \begin{align*}
     \lambda_1(S)&=\frac{A+B+\sqrt{(A-B)^2+4pm}}{2}\\
     &=\frac{A+B+\sqrt{(\alpha-1)^2m^2-2(\alpha-1)m(p-2-c)+4pm+O(1)}}{2}\\
     &=\frac{A+B+\sqrt{((\alpha-1)m+\frac{2p}{\alpha-1}-(p-2-c)+O(\frac{1}{m}))^2}}{2}\\
     &=\alpha m+p+\frac{p}{\alpha-1}+c+O(\frac{1}{m}).
 \end{align*}
    Combining $q(G)\leq \lambda_1(S)$ and Lemma \ref{l5}, we have  $q(G)\leq q(T_r(m+p))+c+p\beta+O(\frac{1}{m})$.
    \end{proof}

\begin{definition}\label{d1}
Let $G$ be a graph, let $t \ge 1$, and let $r \ge 3$ be fixed. 
Define $\mathcal{M}_t(G)$  as the family of all collections
$M=\{\, V(B_1), V(B_2), \dots, V(B_t) \} $, $U(M)=\bigcup\limits_{V(B)\in M}V(B)$ and $C_t(G)=\bigcap\limits_{M\in \mathcal{M}_t(G)}U(M)$ 
,where $B_1, \dots, B_t$ are pairwise vertex‑disjoint copies of $K_{r+1}$ in $G$. 
\end{definition}

\begin{lemma}\label{l9}
    Let $t\geq1, r\ge3$ be fixed integers and $G$ be a $(t+1)K_{r+1}$-free graph. If $\mathcal{M}_t(G)\neq \emptyset$ and $C_t(G)=\emptyset$, then $G$ is $F_{t(r+1)+1,r+1}$-free.
\end{lemma}
\begin{proof}
Suppose to the contrary that $G$ contains a copy of $F_{t(r+1)+1,r+1}$, say $F$. Let $v\in V(F)$ such that $d_F(v)=|V(F)|-1$, and let $Q_1,\ldots,Q_{t(r+1)+1}$ be $t(r+1)+1$ distinct  copies of $K_{r+1}$ in  $F$. Moreover, let $V(Q_i)=\{v\}\cup A_i$ for each $i\in [t(r+1)+1]$. Then $|A_i|=r$ and $A_i\cap A_j=\emptyset$ for $1\leq i\neq j\leq t(r+1)+1$. 

Since $C_t(G)=\emptyset$, there exists $M_1\in \mathcal{M}_t(G)$ such that $v\notin V(B)$ for any $V(B)\in M_1$. Without loss of generality, suppose that $M_1=\{V(B_1),\ldots,V(B_t)\}$, where $B_1,\ldots,B_t$ are $t$ pairwise vertex-disjoint copies of $K_{r+1}$. Then $|\bigcup\limits_{i=1}^{t}V(B_i)|=t(r+1)$, and thus there exists $j\in [t(r+1)+1]$ such that  $(\bigcup\limits_{i=1}^{t}V(B_i))\cap A_j=\emptyset$. Since $v \notin \bigcup_{i=1}^{t} V(B_i)$, the copies $B_1, \dots, B_t, Q_j$ are $t+1$ pairwise vertex-disjoint copies of $K_{r+1}$ in $G$, contradicting that $G$ is $(t+1)K_{r+1}$-free.

This completes the proof.
\end{proof}

The following conclusion is so obvious that we omit its proof.
\begin{Observation}\label{o1}
    Let $r\geq 3$,  $t\geq0$, $k\geq t+1$, and  $n\ge k(r+1)$ be fixed integers. Let $G$ be a graph of order $n$.  Suppose that $G=K_t\vee H$, where $H$ is $(k-t)K_{r+1}$-free. Then $G$ is $kK_{r+1}$-free.
\end{Observation}

 \begin{lemma}\label{ll3}
    Let $r\geq 3$, $t\geq0$, and $n\ge (t+1)(r+1)$ be fixed integers, and let  $G\in Ex_{ssp}(n,(t+1)K_{r+1})$. Then $G$ is connected.
 \end{lemma}
 \begin{proof}
 Suppose to the contrary that $G$ is disconnected.   Let $G_1$ and $G_2$ be two connected components of $G$ with $q(G_1)=q(G)$, and let $G'$ be the graph obtained from $G$ by adding an edge $e$ between $G_1$ and $G_2$.  Then $G'$ is $(t+1)K_{r+1}$-free and $q(G')=q(G_1\cup G_2+e)> q(G_1\cup G_2)=q(G)$ by Lemma \ref{ll1}, which contradicts $G\in Ex_{ssp}(n,(t+1)K_{r+1})$.
 \end{proof}

    \noindent\textbf{\textit{Proof of Theorem \ref{t}.}} We proceed by induction on $t$.

    If $t=0$, then the result holds by Lemma \ref{ll2}. Assume $t \ge 1$ and that the result holds for all non-negative integers less than $t$. We now prove that the result also holds for $t$, thereby completing the proof.

    Clearly, $ K_t\vee T_r(n-t)$ is $(t+1)K_{r+1}$-free. So $q(G)\ge q(K_t\vee T_r(n-t))=q(T_r(n))+t\beta+O(\frac{1}{n})$ since  $G\in Ex_{ssp}(n,(t+1)K_{r+1})$ and Lemma \ref{l6}. Furthermore, $G$ is connected by Lemma \ref{ll3}.

    \noindent\textbf{Claim 1.} For each $p\in\{0\}\cup[t]$, there exists a vertex set $V_p\subseteq V(G)$ with $|V_p|=p$ such that  $H_p=G- V_p$  is $(t+1-p)K_{r+1}$-free and $G=K_p\vee H_p$.
    \begin{proof}
    We proceed by induction on $p$.
    
        If $p=0$, then the claim is trivial. Fix $p\in \{0,\ldots,t-1\}$. 
        Assume that the claim holds for  $p$. We now prove that the claim also holds for $p+1$. 

       Set $s=t+1-p\ge2$ and $n_0=|V(H_p)|=n-p$. By the induction hypothesis, there exists a vertex set $V_p\subseteq V(G)$ with $|V_p|=p$ such that  $H_p=G-V_p$ is $sK_{r+1}$-free and $G=K_p\vee H_p$.  Next we prove that there exists a vertex $u \in V(H_p)$ such that $H_p-\{u\}$ is $(s-1)K_{r+1}$-free and $d_G(u) = n-1$.

       We claim that there exist $s-1$ pairwise vertex-disjoint copies of $K_{r+1}$ in $H_p$. Otherwise,  $H_p$ is $(s-1)K_{r+1}$-free. By the induction hypothesis on $t$ and Lemma \ref{l6}, we have $q(H_p)\leq q(K_{s-2}\vee T_r(n_0-s+2))=\alpha n_0+(s-2)\beta+O(\frac{1}{n_0})$, and thus $q(G)=q(K_p\vee H_p)\leq q(T_r(n))+(t-1)\beta+O(\frac{1}{n})$ by Lemma \ref{l8}, which contradicts $q(G)\ge q(T_r(n))+t\beta+O(\frac{1}{n})$. Combining Definition \ref{d1}, we have $\mathcal{M}_{s-1}(H_p)\neq \emptyset$.

Next we prove that $C_{s-1}(H_p)\neq\emptyset$. Otherwise, $C_{s-1}(H_p)=\emptyset$, and thus $H_p$ is $F_{(s-1)(r+1)+1,r+1}$-free by Lemma \ref{l9}. By Lemma \ref{l7} and Lemma \ref{l8}, we have $q(G)=q(K_p\vee H_p)\leq q(T_r(n))+p\beta+O(\frac{1}{n})$, which contradicts $q(G)\ge q(T_r(n))+t\beta+O(\frac{1}{n})$.

Without loss of generality, let $u\in C_{s-1}(H_p)$. Then $H_p-\{u\}$ is $(s-1)K_{r+1}$-free.

Next we prove that $d_G(u) = n-1$. Otherwise, we add all missing edges between $u$ and $H_p-\{u\}$, and denote the resulting graph by $H^+_p$. Then $K_p\vee H^+_p\cong K_{p+1}\vee (H_p-\{u\})$ is $(t+1)K_{r+1}$-free by Observation \ref{o1} and $s=t+1-p$. However, by Lemmas \ref{ll3} and \ref{ll1}, we have $q(K_p\vee H^+_p) > q(K_p\vee H_p) = q(G)$, which contradicts the fact that $G \in \mathrm{Ex}_{ssp}(n, (t+1)K_{r+1})$.

Therefore, $G=K_{p+1}\vee (H_p-\{u\})$. Let $V_{p+1}=V_p\cup \{u\}$ and $H_{p+1}=H_p-\{u\}$. Then $H_{p+1}$ is $(t-p)K_{r+1}$-free. This completes the proof of Claim 1.
    \end{proof}

By Claim 1, there exists a vertex set $V_t\subseteq V(G)$ with $|V_t|=t$ such that  $H_t=G- V_t$  is $K_{r+1}$-free and $G=K_t\vee H_t$. Repeatedly applying Lemma \ref{ll6}, together with $G \in \mathrm{Ex}_{ssp}(n, (t+1)K_{r+1})$, we obtain that $H_t$ is a complete $r'$-partite graph, where $r' \le r$. For sufficiently large $n$, there exist two vertices $u$ and $v$ in the same part of $H_t$. If $r'<r$, then $H_t+uv$ is $K_{r+1}$-free, and thus $G+uv$ is $(t+1)K_{r+1}$-free. Clearly, we have $q(G+uv)>q(G)$, which contradicts the fact that $G \in \mathrm{Ex}_{ssp}(n, (t+1)K_{r+1})$. 
Thus we have $r'=r$. Without loss of generality, let $H_t\cong K_{b_1,\ldots,b_r}$ and $G\cong K_t\vee K_{b_1,\ldots,b_r}$, where $\sum\limits_{i=1}^{r}b_i=n-t$ and $b_i\geq1$ for each $i\in[r]$. By the proof of  Lemma \ref{l6}, we have \begin{align}
    1=\frac{t}{q(G)-n+2}+\sum\limits_{i=1}^{r}\frac{b_i}{q(G)-n+2b_i}.\label{s4}
\end{align}

Now we claim that $G\cong K_t\vee T_r(n-t)$. Otherwise, there exist $i,j\in [r]$ such that $b_i-b_j\geq 2$. Without loss of generality, let $i=1,j=2$,  $f(x)=\frac{x}{q(G)-n+2x}$ and $x>0$. Then we have $f'(x)=\frac{q(G)-n}{(q(G)-n+2x)^2}$ and $f''(x)=\frac{-4(q(G)-n)}{(q(G)-n+2x)^3}<0$, and thus $f(x)$ is a concave function. This implies \begin{align}
    f(b_1-1)+f(b_2+1)>f(b_1)+f(b_2).\label{s5}
\end{align} Let $G'\cong K_t\vee K_{b_1-1,b_2+1,b_3,\ldots,b_r}$. Then $G'$ is $(t+1)K_{r+1}$-free and 
\begin{align}
    1=\frac{t}{q(G')-n+2}+\frac{b_1-1}{q(G')-n+2(b_1-1)}+\frac{b_2+1}{q(G')-n+2(b_2+1)}+\sum\limits_{i=3}^{r}\frac{b_i}{q(G')-n+2b_i}.\label{s6}
\end{align}
From equations \eqref{s4}, \eqref{s5} and \eqref{s6}, we have $q(G')>q(G)$, which contradicts the fact that $G \in \mathrm{Ex}_{ssp}(n, (t+1)K_{r+1})$.

Therefore, we have $G\cong K_t\vee T_r(n-t)$. $\hfill\square$
    
	\section*{\bf Funding}
	
 This work is  supported by the National Natural Science Foundation of China (Grant Nos. 12371347, 12271337).
	
	\section*{Declarations}
	
	\noindent\textbf{Conflict of interest}\  The authors declare that they have no conflict of interest.
	
	\vskip 0.5em
	
	\noindent\textbf{Data availability} \  No data was used for the research described in the article.

	\end{spacing}
\end{document}